\documentclass[11pt]{amsart}
\usepackage{amsfonts}
\usepackage{amsmath}
\usepackage{amssymb}
\usepackage{makecell}
\usepackage{mathrsfs}
\usepackage[utf8]{inputenc}

\usepackage{amsthm,amsfonts,latexsym,epsfig,geometry}
\usepackage{hyperref,times}
\usepackage{amscd}
\usepackage{mathtools}
\usepackage{color}

\allowdisplaybreaks[1]
\newtheorem{theorem}{Theorem}[section]
\newtheorem{lemma}[theorem]{Lemma}

\newtheorem{corollary}[theorem]{Corollary}
\theoremstyle{definition}
\newtheorem{definition}[theorem]{Definition}

\theoremstyle{remark}

\numberwithin{equation}{section}
\usepackage{hyperref}
\hypersetup{
	colorlinks=true,
	linkcolor={blue}, urlcolor={blue},
	citecolor={red}, anchorcolor = {blue}}
\begin{document}
\title[ ]{Asymptotic and Arithmetic Aspects of Generalized Colored Overpartitions}
\author{ANAKHA V and CHIRANJIT RAY }
\address{Anakha V, National Institute of Technology Calicut, Kerala 673601, INDIA}
\curraddr{}
\email{anakha$\textunderscore$p240019ma@nitc.ac.in }
\address{Chiranjit Ray, National Institute of Technology Calicut, Kerala 673601, INDIA}
\curraddr{}
\email{chiranjitray@nitc.ac.in }
\date{\today}
\thanks{} 
\begin{abstract}

Colored partitions constitute a well-established and extensively studied area of research with numerous generalizations arising from imposing restrictions on the allowed colors of the parts. In particular, restricted colored partitions have received considerable attention owing to their rich combinatorial and arithmetic properties. Motivated by these developments, we consider a colored overpartition function $\overline{\mathfrak{C}}_{m,r,s}(n)$, which enumerates the overpartitions of $n$ such that each part divisible by $m$ can be assigned any of $r$ colors, while each part not divisible by $m$ may be assigned any of $s$ colors. We investigate an asymptotic formula and several arithmetic properties of this function, including its density properties.

\end{abstract}
\maketitle
\section{Introduction and statement of results}
\label{intro}
 A \emph{partition} of a nonnegative integer $n$ is defined by a sequence of positive integers $(\lambda_1,\lambda_2,\hdots,\lambda_k)$ with $\lambda_1\geq \lambda_2\geq \cdots \geq \lambda_k$ such that $n=\lambda_1+\lambda_2+\cdots+\lambda_k$. The integers $\lambda_i$ are called \emph{parts} of the partition. Let $p(n)$ denote the number of partitions of $n$, where $p(0):=1$, by convention. Its generating function is given by
 \begin{align*}
    \sum_{n=0}^{\infty}{p}(n)q^n=\prod_{n=1}^\infty\frac{1}{1-q^n}=\frac{1}{(q;q)_\infty}.
\end{align*}

The following standard $q$-series notation will be employed throughout this paper, with $q \in \mathbb{C}$ satisfying $|q| < 1$:

\begin{equation*}
\begin{split}
    (a;q)_0&:= 1 \\
     (a;q)_n&:=   \prod_{k=0}^{n-1}(1-aq^k) ~\text{ for }n\geq 1, \text{ and }\\
    (a;q)_\infty&:=\prod_{k=0}^{\infty}(1-aq^k).
\end{split} 
\end{equation*}
Closely connected with these products is Ramanujan's general theta function $f(a,b)$, which is defined by \cite[pp. 34, 18.1]{ramanujan3},

\begin{align*}
    f(a,b)=\sum_{n=-\infty}^{\infty}a^{n(n+1)/2}b^{n(n-1)/2}=(-a;ab)_\infty(-b;ab)_\infty(ab;ab)_\infty, \text{ for }|ab|<1.
\end{align*}
A notable special case is the theta function
 \begin{align*} 
     \varphi(q):=f(q,q)=1+2\sum_{n=1}^{\infty}q^{n^2}=\frac{(-q;q^2)_\infty(q^2;q^2)_\infty}{(q;q^2)_\infty(-q^2;q^2)_\infty}=\frac{(q^2;q^2)_\infty^5}{(q;q)_\infty^2 (q^4;q^4)_\infty^2}.
 \end{align*}
Substituting $q$ with $-q$ in the above identity, we obtain

\begin{align*}
    \varphi(-q)=\frac{(q;q)_\infty^2}{(q^2;q^2)_\infty}.
\end{align*}

An \emph{overpartition} of $n$ is a partition of $n$ in which the first occurrence of each part may be overlined.
The concept of overpartitions was introduced by Corteel and Lovejoy \cite{corteel2004overpartitions}. Let $\overline{p}(n)$ denote the number of overpartitions of $n$ with the convention that $\overline{p}(0):=1$. For example, $\overline{p}(4)=14$ and the corresponding overpartitions are 
\begin{align*}
    (4),(\overline{4}), (3,1),(\overline{3},1),(3,\overline{1}),(\overline{3},\overline{1}),(2,2), (\overline{2},2), (2,1,1),(\overline{2},1,1), (2,\overline{1},1), (\overline{2},\overline{1},1), (1,1,1,1), (\overline{1},1,1,1).
\end{align*}
The generating function for $\overline{p}(n)$ is given by
\begin{align*}
    \sum_{n=0}^{\infty}\overline{p}(n)q^n=\prod_{n=1}^\infty\left(\frac{1+q^n}{1-q^n}\right)=\frac{(q^2;q^2)_\infty}{(q;q)_\infty^2}.
\end{align*}

For a positive integer $k$, let $p_k(n)$ denote the number of $k$-colored partitions of $n$, which is defined as the partitions of $n$ in which each part can appear in $k$ colors. The generating function for $p_k(n)$ is given by
\[
\sum_{n=0}^{\infty} p_k(n)q^n
=\frac{1}{(q;q)_\infty^{k}}.
\]
Using Euler's and Jacobi's identities, Gandhi \cite{gandhi1963congruences} obtained several infinite families of colored partition congruences, including the following:
\begin{align*}
   & p_2(5n+3)\equiv 0\pmod{5}\\
   & p_8(11n+4)\equiv 0\pmod{11}.
\end{align*}

Agarwal and Andrews \cite{agarwal1987rogers} initiated an investigation of colored partitions within the framework of unrestricted partitions. Since then, numerous variants and extensions of colored partitions have been introduced and studied from a variety of combinatorial and arithmetic perspectives. For instance, Merca and Simion \cite{merca2023n} studied the number of $n$-color partitions of $n$ into distinct parts which have a number of parts congruent to $r$ modulo $2$, where the $n$-color partition is a partition in which a part of size $n$ can come in $n$ different colors. Baruah and Das \cite{baruah2022generating} investigated the $9$-regular and $27$-regular partitions in $3$ colors.

In \cite{thejitha2026arithmetic}, Thejitha, Sellers, and Fathima studied colored partitions in which even parts admit $r$ possible colors, whereas odd parts admit $s$ possible colors, for fixed integers $r,s \geq 1$. They denoted the number of such partitions of $n$ by $a_{r,s}(n)$. The generating function associated with $a_{r,s}(n)$ is
\[
\sum_{n=0}^{\infty} a_{r,s}(n)q^n
=\frac{(q^2;q^2)_\infty^{s-r}}{(q;q)_\infty^{s}}.
\]

The parameter choices $r=1$ and $s=1$ were introduced independently in the works of Hirschhorn and Sellers \cite{hirschhorn2025family} and Amdeberhan, Sellers, and Singh \cite{amdeberhan2025arithmetic}, respectively.

Recently, Thejitha and Fathima \cite{thejitha2026overcolored} defined the function $\overline{a}_{r,s}(n)$ to enumerate number of overpartitions of $n$ in which each even part may be assigned one of $r$ colors and each odd part may be assigned one of $s$ colors, where $r,s \geq 1$ are fixed integers. As in ordinary overpartitions, the first occurrence of any distinct part is allowed to be overlined. They proved that the generating function for $\overline{a}_{r,s}(n)$ is
\[
\sum_{n=0}^{\infty}\overline{a}_{r,s}(n)q^n
=\frac{(q^2;q^2)_\infty^{3s-2r}}
{(q;q)_\infty^{2s}(q^4;q^4)_\infty^{s-r}}.
\]

In the same work, the authors showed that $\overline{a}_{r,s}(n)$ satisfies several infinite families of congruences, including numerous Ramanujan-type congruences modulo primes $p\geq 3$ and modulo powers of $2$. 

In our work, we generalize the function $\overline{a}_{r,s}(n)$ by replacing the parity condition with divisibility by an arbitrary integer $m\geq 2$. Specifically, for a fixed integer $m\geq 2$, let $\overline{\mathfrak{C}}_{m,r,s}(n)$ denote the number of overpartitions of $n$ in which each part divisible by $m$ can occur in any of $r$ colors, whereas each part not divisible by $m$ can occur in any of $s$ colors. We obtain the following generating function:
\begin{align}\label{gen}
    \sum_{n=0}^{\infty}\overline{\mathfrak{C}}_{m,r,s}(n)q^n=\frac{(q^2;q^2)_\infty^{s}(q^{2m};q^{2m})_{\infty}^{r-s}}{(q;q)_\infty^{2s}(q^m;q^m)_\infty^{2(r-s)}}.
\end{align}
As an example, if $m=3,~r=2$, and $s=1$, then $\overline{\mathfrak{C}}_{3,2,1}(3)=10$, and the overpartitions associated with this are:
\[(3_1),(3_2),(\overline{3}_1),(\overline{3}_2),(2_1,1_1),(\overline{2}_1,1_1),(2_1,\overline{1}_1),(\overline{2}_1,\overline{1}_1),(1_1,1_1,1_1),(\overline{1}_1,1_1,1_1).\]
Clearly, when $m=2$, the function $\overline{\mathfrak{C}}_{m,r,s}(n)$ reduces to $\overline{a}_{r,s}(n)$. Furthermore, if we set $m=2$ and $s=1$, $\overline{\mathfrak{C}}_{m,r,s}(n)$ coincide with $\overline{a}_c(n)$, the generalized overcubic partitions of $n$, which is considered in \cite{amdeberhan2025arithmetic}. In the same paper it is proved that for a prime $p\geq 3$ and $r$, $1\leq r\leq p-1$, such that $r$ is a quadratic nonresidue modulo $p$,
\[\overline{a}_{kp-1}(pn+r)\equiv 0\pmod{p},\]
for all $n\geq 0$, where $k$ is a positive integer.

In \cite{andrews2015singular}, Andrews defined the singular overpartition function $\overline{C}_{k,i}(n)$, which counts the number of overpartitions of $n$ having no parts divisible by $k$, with overlining is permitted only for parts congruent to $\pm i$ modulo $k$. The corresponding generating function is
 \begin{align*}
     \sum_{n=0}^{\infty}\overline{C}_{k,i}(n)q^n=\frac{(q^k;q^k)_{\infty}(-q^i;q^k)_{\infty}(-q^{k-i};q^k)_{\infty}}{(q;q)_{\infty}}
 \end{align*}
 for $k\geq 3$ and $1\leq i\leq \lfloor \frac{k}{2}\rfloor$. In the same article, Andrews proved the following congruences:
 \begin{align*}
   \overline{C}_{3,1}(9n+3)\equiv \overline{C}_{3,1}(9n+6)\equiv 0 \pmod 3 \text{ \hspace{1cm} for all $n\geq 0$ }.  
 \end{align*}
 The study of Andrews' singular overpartition has become an active area of research. For instance, Chen \textit{et. al} \cite{chen2015arithmetic} studied the arithmetic properties of $\overline{C}_{3,1}(n)$, $\overline{C}_{4,1}(n)$, $\overline{C}_{6,1}(n)$, and $\overline{C}_{6,2}(n)$. Barman and Singh \cite{singh2021divisibility} studied the divisibility of $\overline{C}_{6,2}(n)$ and $\overline{C}_{12,4}(n)$ by powers of $2$ and $3$. In \cite{ray2024distribution}, the distribution of $\overline{C}_{p,1}(n)$ was studied for primes $p\geq 5$. 

 We obtain a relation between $\overline{\mathfrak{C}}_{3,1,2}(n)$ and $\overline{C}_{3,1}(n)$ as stated below. 
 
 \begin{theorem}\label{sing}
 For every $n\geq 1$, we have
      \begin{align*}
         \overline{C}_{3,1}(n)= \overline{\mathfrak{C}}_{3,1,2}(n)+2\sum_{k=1}^{\lfloor \sqrt{n}\rfloor}(-1)^k~ \overline{\mathfrak{C}}_{3,1,2}(n-k^2).
     \end{align*}
 \end{theorem}

In this paper, we undertake a systematic study of the arithmetic properties of $\overline{\mathfrak{C}}_{m,r,s}(n)$ and establish the following congruences modulo $4$.

\begin{theorem}\label{thm1}
    Let $m\geq 2$ be an integer. Let $r,s\geq 1$ be fixed. If $r$ and $s$ are both even, then for every $n\geq 1$,

\begin{align*}
    \overline{\mathfrak{C}}_{m,r,s}(n)\equiv 0 \pmod{4}.
\end{align*}
If $r$ and $s$ are odd, then for every $n\geq 1$,
\begin{align*}
    \overline{\mathfrak{C}}_{m,r,s}(n)\equiv 
    \begin{cases}
        2s, \hspace{0.85cm} \text{if $n=k^2$ for some $k\geq 1$ };\\
        0, \hspace{1cm} \text{otherwise}
    \end{cases}\pmod{4}.
\end{align*}
If $r$ is odd and $s$ is even, then for every $n\geq 1$,
\begin{align*}
    \overline{\mathfrak{C}}_{m,r,s}(n)\equiv 
    \begin{cases}
        2(r-s), \hspace{1cm} \text{if $n=mk^2$ for some $k\geq 1$};\\
        0, \hspace{2.1cm} \text{otherwise}
    \end{cases}\pmod{4}.
\end{align*}
If $r$ is even and $s$ is odd, then for every $n\geq 1$,
\begin{align*}
    \overline{\mathfrak{C}}_{m,r,s}(n)\equiv 
    \begin{cases}
        2s, \hspace{2cm} \text{if $n=k^2$ for some $k\geq 1$ and $m$ is not a square};\\
        2r, \hspace{2cm} \text{if $n=mk^2$ for some $k\geq 1$ and if $m$ is a square};\\
        2(r-s), \hspace{1.05cm}  \text{if $n=mk^2$ for some $k\geq 1$ and if $m$ is not a square};\\
        0, \hspace{2.17cm} \text{otherwise}
    \end{cases}\pmod{4}.
\end{align*}
    
\end{theorem}

We further investigate the arithmetic behavior of $\overline{\mathfrak{C}}_{m,r,s}(n)$ by establishing congruence relations modulo an arbitrary prime $p$.

\begin{theorem}\label{modp}
    Suppose $p$ is a prime such that $\gcd(m,p)=1$ and $\left( \frac{m^{-1}r}{p} \right)=-1$, where $\left( \frac{\cdot}{p} \right)$ denotes the Legendre symbol. Then for every $n\geq 1$, $\overline{\mathfrak{C}}_{m,p-1,p}(pn+r)\equiv0\pmod{p}.$
\end{theorem}

In addition, we derive an infinite family of congruences modulo $8$, thereby extending our investigation of the arithmetic properties of $\overline{\mathfrak{C}}_{m,r,s}(n)$.

\begin{theorem}\label{Theorem4}
    If $p_i$ are primes such that $p_i\equiv 3 \pmod{4}$ for $1\leq i \leq k+1$, then for any integer $j\not\equiv 0 \pmod{p_{k+1}},$ and any positive integer $\alpha$, we have
    \begin{align}\label{EQ1}
      \overline{\mathfrak{C}}_{2^\alpha,1,2}(4p_1^2\cdots p_{k+1}^2n+(4j+p_{k+1})p_1^2\cdots p_k^2 p_{k+1}) \equiv 0 \pmod{8}.
    \end{align}
\end{theorem}

The study of arithmetic densities associated with partition functions is also an active area of research. In \cite{singh2021certain} and \cite{singh2021new}, the authors established an arithmetic density result for Andrews' singular overpartition function.
We next turn our attention to the arithmetic distribution of $\overline{\mathfrak{C}}_{m,r,s}(n)$. In particular, we determine the density of the set of positive integers $n$ for which $\overline{\mathfrak{C}}_{m,r,s}(n)$ satisfies the arithmetic condition described below.

\begin{theorem}\label{arith.dens}
    Let $p$ be a prime number greater than $3$ and $r$ be a natural number such that $p-(r+1)$ is an even positive integer. Then for every $k\geq 1$,
     \begin{align*}
         \lim_{X\to\infty}\frac{\#\{0<n\leq X: \overline{\mathfrak{C}}_{p,r,p}(n)\equiv 0 \pmod{p}\}}{X}=1.
    \end{align*}
\end{theorem}

\begin{theorem}\label{dens.op}
    Fix an integer $m\geq 2$. Suppose that $p$ is a prime, and let $p^\alpha\mid r$ and $p^\alpha\mid s$ for some positive integer $\alpha$. Then for every $n\geq 0$,
    \begin{align*}
        \overline{\mathfrak{C}}_{m,r,s}(p^\alpha n)\equiv \overline{\mathfrak{C}}_{m,\frac{r}{p^\alpha},\frac{s}{p^\alpha}}(n)\pmod{p}.
    \end{align*}
    Moreover,
    \begin{align*}
         \lim_{X\to\infty}\frac{\#\{0<n\leq X: \overline{\mathfrak{C}}_{m,r,s}(n)\equiv 0 \pmod{p}\}}{X}\geq 1-\frac{1}{p}.
    \end{align*}
\end{theorem}

Asymptotic formulas for partition functions and their various generalizations have been extensively investigated.
The asymptotic growth of the ordinary partition function was first determined by Hardy and Ramanujan \cite{ramanujan1918asymptotic}, who showed that
\begin{align*}
    p(n) \sim \frac{1}{4\sqrt{3}\,n}\, \exp\!\left(2\pi\sqrt{\frac{n}{6}}\right)
    \qquad \text{as } n \to \infty.
\end{align*}
Chen~\cite{chen2016congruences} extended the asymptotic study to the singular overpartition function, deriving an asymptotic expression for its coefficients. The asymptotic behavior of the sums of odd and even minimal excludants associated with the partitions of $n$ was studied in \cite{barman2024arithmetic}. Several authors have obtained asymptotic formulas for colored partition functions. For example, in \cite{guadalupe2024note} an asymptotic formula for the number of partitions of $n$ where parts that are multiple of $p$ come up with $2$ colors is obtained. These results naturally lead us to consider the asymptotic behavior of $\overline{\mathfrak{C}}_{m,r,s}(n)$ and obtain the following asymptotic formula.

\begin{theorem}\label{thm2}
    As $n\to \infty$,
    \begin{align*}
        \overline{\mathfrak{C}}_{m,r,s}(n)\sim \frac{m^{\frac{r-s}{2}}\left(\frac{r-s}{m}+s\right)^{\frac{r+1}{4}}}{2^{\frac{3(r+1)}{2}}n^{\frac{r+3}{4}}}\exp\left(\pi \sqrt{\left( \frac{r-s}{m}+s \right)n} \right).
    \end{align*}
\end{theorem}

If we set $m=2$ and $s=1$ in the above theorem, we immediately obtain the following corollary for the generalized overcubic partition function, which was studied in~\cite{amdeberhan2025arithmetic}.

\begin{corollary}
     As $n\to \infty$,
    \begin{align*}
       \overline{a}_c(n)\sim \frac{\left(\frac{c+1}{2}\right)^{\frac{c+1}{4}}}{2^{c+2}n^{\frac{c+3}{4}}}\exp\left(\pi \sqrt{\left( \frac{c+1}{2} \right)n} \right).
    \end{align*}
\end{corollary}

\noindent The remainder of this paper is structured as follows. In Section \ref{section1}, we recall the necessary preliminaries and auxiliary results required for the proofs. The proofs of Theorems \ref{sing}, \ref{thm1}, and \ref{modp} are given in Section \ref{section2}. Subsequently, Section \ref{section3} is devoted to prove Theorems \ref{Theorem4}, \ref{arith.dens}, and \ref{dens.op}. Finally, in Section \ref{section4}, we establish Theorem \ref{thm2}, which completes the proof of the remaining principal result of the paper.

\section{Preliminaries}\label{section1}

We recall several notions and results from the theory of modular forms that will be used in the sequel. We refer to \cite{koblitz2012introduction,ono2004web} for further discussion and details. Let us denote the complex upper half-plane by $\mathcal{H}$. For a congruence subgroup $\Gamma$, let $M_w(\Gamma)$ denote the complex vector space of modular forms of positive integral weight $w$ on $\Gamma$. For a positive integer $N$, the congruence subgroup of level $N$, denoted by $\Gamma_0(N)$, is defined by
\begin{align*}
    \Gamma_0(N)=
\left\{
\begin{pmatrix}a&b\\c&d\end{pmatrix}\in SL_2(\mathbb Z):
c\equiv0\pmod N\right\}.
\end{align*}

\begin{definition}[ Definition $1.15$, \cite{ono2004web}]
     Let $\chi$ be a Dirichlet character modulo N. A modular form $f(z) \in M_w(\Gamma_1(N)) $ is said to have Nebentypus character $\chi$ if
        \begin{align*}
            f\left(\frac{az+b}{cz+d}\right)=\chi(d)(cz+d)^wf(z)
        \end{align*}
        for all $z\in\mathcal{H}$ and all 
        $\begin{pmatrix}
            a & b\\
            c & d
        \end{pmatrix}\in \Gamma_0(N).$ The space of such modular forms is denoted by $M_w(\Gamma_0(N),\chi)$.
\end{definition}

A class of modular forms of particular relevance to this work is furnished by eta-quotients. Recall that the Dedekind eta function, $\eta(z)$, defined by
\begin{align*}
    \eta(z):=q^{1/24}(q;q)_{\infty},
\end{align*}
where $q=e^{2\pi iz}$ and $z\in\mathcal{H}.$ An eta-quotient is a function of the form
\begin{align*}
    f(z)=\prod_{\delta\mid N}\eta(\delta z)^{r_\delta},
\end{align*}
where $N \geq 1$ and each $r_\delta$ is an integer. 

\begin{theorem}[Theorem $1.64$, \cite{ono2004web}]\label{RTHM1}
   If $f(z)=\prod\limits_{\delta\mid N} \eta(\delta z)^{r_\delta}$ is an eta-quotient such that
   \begin{align*}
       w &= \frac{1}{2} \sum\limits_{\delta\mid N}r_{\delta}\in \mathbb{Z},\\
       \sum\limits_{\delta\mid N}\delta r_{\delta} &\equiv 0 \pmod{24},
        \end{align*}
        and
       \begin{align*}      
       \sum\limits_{\delta\mid N}\frac{N}{\delta} r_{\delta} &\equiv 0 \pmod{24},
   \end{align*}
   then 
   \begin{align*}
       f\left(\frac{az+b}{cz+d}\right)=\chi(d) (cz+d)^w f(z)
   \end{align*}
   for every $\begin{pmatrix}
       a & b\\
       c & d
   \end{pmatrix}\in \Gamma_0(N)$. Here
   \begin{align*}
       \chi(d):=\left( \frac{(-1)^w \prod\limits_{\delta\mid N}\delta^{r_\delta}}{d}\right).
   \end{align*}
\end{theorem}

\noindent If $f(z)$ is an eta-quotient satisfying
the conditions of Theorem \ref{RTHM1}, and if $f(z)$ is holomorphic at all of the cusps of $\Gamma_0(N)$, then $f(z) \in M_w(\Gamma_0(N),\chi).$ To check whether $f(z)$ is holomorphic at the cusps, it suffices to check that the orders at the cusps are non-negative. The following theorem is the necessary criterion for determining orders of an eta-quotient at cusps.

\begin{theorem}[Theorem $1.65$, \cite{ono2004web}]\label{RTHM2}
    Let $c,d,$ and $N$ be positive integers with $d\mid N$ and $\gcd(c,d)=1.$ If $f(z)$ is an eta-quotient satisfying the conditions of Theorem \ref{RTHM1} for $N$, then the order of vanishing of $f(z)$ at the cusp $\dfrac{c}{d}$ is 
    \begin{align*}
        \frac{N}{24}\sum_{\delta\mid N}\frac{\gcd(d,\delta)^2r_\delta}{\gcd(d,\frac{N}{d})d\delta}.
    \end{align*}
\end{theorem}

\noindent Hecke operators play a fundamental role in the theory of modular forms and act naturally as linear operators on spaces of modular forms. We begin by recalling their definition and basic properties in the setting of modular forms of integer weight.

\begin{definition}
    If $m$ is a positive integer and $f(z)=\sum\limits_{n=0}^{\infty}a(n)q^n\in M_w(\Gamma_{0}(N),\chi)$,  then the action of
the Hecke operator $T_m$ on $f(z)$ is defined by
\begin{align*}
    f(z)\mid T_m:=\sum\limits_{n=0}^{\infty}\left( \sum\limits_{d\mid \gcd(n,m)}\chi(d)d^{w-1}a\left(\frac{nm}{d^2}\right) \right)q^n.
\end{align*}
If $m=p$ is a prime number, then
\begin{align}\label{H1}
    f(z)\mid T_p=\sum\limits_{n=0}^{\infty}\left( a(pn)+\chi(p)p^{w-1}a\left(\frac{n}{p}\right) \right)q^n.
\end{align}
\end{definition}

\begin{definition}
     A modular form $f(z)\in M_w(\Gamma_{0}(N),\chi)$ is called a Hecke eigen form if for every $m \geq 2$ there is a complex number $\lambda(m)$ for which
     \begin{align}\label{H2}
         f(z)\mid T_m=\lambda(m)f(z).
     \end{align}
\end{definition}

We shall make use of the following theorem of Serre in establishing one of the principal results of this paper.
\begin{theorem}[Theorem $2.65$, \cite{ono2004web}]\label{RTHM3}
Let $A$ denote the subset of integer weight modular forms in $M_w(\Gamma_{0}(N),\chi)$ whose Fourier coefficients are in $\mathcal{O}_{K}$, the ring of algebraic integers in a number field $K$. Suppose $\mathcal{M}\subset \mathcal{O}_{K}$ is an ideal. If $f(z)\in A$ has a Fourier expansion 
\begin{align*}
    f(z)=\sum_{n=0}^{\infty}a(n)q^n,
\end{align*}
then there is a constant $\alpha>0$ such that
\begin{align*}
    \#\{n\leq X: a(n)\not\equiv 0 \pmod{\mathcal{M}}\}=O\left( \frac{X}{(\log X)^\alpha} \right).
\end{align*}
\end{theorem}


\section{Proof of Theorems \ref{sing}, \ref{thm1}, and \ref{modp}}\label{section2}

 \begin{proof}[Proof of Theorem \ref{sing}]
By manipulating the generating function of $\overline{C}_{3,1}(n)$, we have
    \begin{align*}
        \sum_{n=0}^{\infty}\overline{C}_{3,1}(n)q^n&=\frac{(q^2;q^2)_\infty(q^{3};q^{3})_{\infty}^{2}}{(q^6;q^6)_\infty(q;q)_\infty^{2}}\\ 
  &=\frac{(q;q)_\infty^2}{(q^2;q^2)_\infty}   \sum_{n=0}^{\infty}\overline{\mathfrak{C}}_{3,1,2}(n)q^n\\
       &=\varphi(-q) \sum_{n=0}^{\infty}\overline{\mathfrak{C}}_{3,1,2}(n)q^n\\
       &= \left(1+2\sum_{n=1}^{\infty}(-1)^nq^{n^2}\right) \sum_{n=0}^{\infty}\overline{\mathfrak{C}}_{3,1,2}(n)q^n\\
       &= \sum_{n=0}^{\infty}\overline{\mathfrak{C}}_{3,1,2}(n)q^n+2\sum_{n=0}^{\infty}\sum_{k=1}^{\infty}(-1)^k~\overline{\mathfrak{C}}_{3,1,2}(n)q^{n+k^2}\\
       &= \sum_{n=0}^{\infty}\overline{\mathfrak{C}}_{3,1,2}(n)q^n+2\sum_{n=1}^{\infty}\left( \sum_{k=1}^{\lfloor \sqrt{n}\rfloor}(-1)^k~\overline{\mathfrak{C}}_{3,1,2}(n-k^2)\right)q^n .
    \end{align*}
 Hence Theorem \ref{sing} follows.    
 \end{proof}

 \begin{proof}[Proof of Theorem \ref{thm1}]

     Using the generating function given in (\ref{gen}), we have 
     \begin{align}
    \notag\sum_{n=0}^{\infty}\overline{\mathfrak{C}}_{m,r,s}(n)q^n&=\frac{1}{\varphi(-q)^s \varphi(-q^m)^{r-s}}\\
   \notag &=\varphi(q)^s \varphi(q^2)^{2s}\varphi(q^4)^{4s}\cdots \varphi(q^m)^{r-s}\varphi(q^{2m})^{2(r-s)}\varphi(q^{4m})^{4(r-s)}\cdots\\
 \label{thm1eqn1}   &\equiv \varphi(q)^s \varphi(q^m)^{r-s} \pmod{4}.
\end{align}
To get (\ref{thm1eqn1}), we used the facts that $\dfrac{1}{\varphi(-q)}=\varphi(q)\varphi(q^2)^2\varphi(q^4)^4\cdots$ and $\varphi(q^i)^j\equiv 1 \pmod{4}$ for any integers $i$ and $j\geq2$.\\

\noindent \textbf{Case 1:} $r$ and $s$ are even\\

From (\ref{thm1eqn1}), we can write
\begin{align}
     \notag\sum_{n=0}^{\infty}\overline{\mathfrak{C}}_{m,r,s}(n)q^n&\equiv 1 \pmod{4}.
\end{align}
Thus for every $n\geq 1$, $\overline{\mathfrak{C}}_{m,r,s}(n)\equiv 0 \pmod{4}$.\\

\noindent \textbf{Case 2:} $r$ and $s$ are odd\\

The expression in (\ref{thm1eqn1}) becomes

\begin{align*}
    \sum_{n=0}^{\infty}\overline{\mathfrak{C}}_{m,r,s}(n)q^n&\equiv \notag\varphi(q)^s \pmod{4}\\
    &= \left(1+2\sum_{n=1}^{\infty}q^{n^2}\right)^s\\
    &\equiv 1+2\binom{s}{1} \sum_{n=1}^{\infty}q^{n^2}\\
    &= 1+ 2s \sum_{n=1}^{\infty}q^{n^2} \pmod{4}.    
\end{align*}
By comparing the coefficients we get the congruences in the second case of Theorem \ref{thm1}.\\

\noindent \textbf{Case 3:} $r$ is odd and $s$ is even\\

In this case, (\ref{thm1eqn1}) becomes

\begin{align*}
    \sum_{n=0}^{\infty}\overline{\mathfrak{C}}_{m,r,s}(n)q^n&\equiv \varphi(q^m)^{r-s} \pmod{4}\\
    &\equiv 1+2(r-s)\sum_{n=1}^{\infty}q^{mn^2}\pmod{4}.
\end{align*}
Equating the coefficients gives the congruence in the third case of Theorem \ref{thm1}.\\

\noindent \textbf{Case 4:} $r$ is even and $s$ is odd\\

From (\ref{thm1eqn1}), we get
\begin{align*}
    \sum_{n=0}^{\infty}\overline{\mathfrak{C}}_{m,r,s}(n)q^n&\equiv \varphi(q)^s \varphi(q^m)^{r-s} \pmod{4}\\
    &\equiv\left( 1+2s \sum_{n=1}^{\infty}q^{n^2}\right)\left(1+2(r-s)\sum_{n=1}^{\infty}q^{mn^2}\right)\pmod{4}\\
    &\equiv 1+2s \sum_{n=1}^{\infty}q^{n^2}+2(r-s)\sum_{n=1}^{\infty}q^{mn^2}\pmod{4}.
\end{align*}
The final case is also obtained by equating the coefficients.
 \end{proof}

 \begin{proof}[Proof of Theorem \ref{modp}]
    Substituting $r=p-1$ and $s=p$ in (\ref{gen}) gives, 
\begin{align*}
    \sum_{n=0}^{\infty}\overline{\mathfrak{C}}_{m,p-1,p}(n)q^n=\frac{(q^2;q^2)_\infty^p (q^m;q^m)_\infty^2}{(q;q)_\infty^{2p}(q^{2m};q^{2m})_\infty}.
\end{align*}
   By using binomial theorem, we have $(q^k;q^k)_\infty^{p^\ell}\equiv (q^{pk};q^{pk})_\infty^{p^{\ell-1}}\pmod{p^\ell}$ for any prime $p$ and positive integers $k$ and $\ell$. Thus,  
   \begin{align*}
    \sum_{n=0}^{\infty}\overline{\mathfrak{C}}_{m,p-1,p}(n)q^n&\equiv\frac{(q^{2p};q^{2p})_\infty (q^m;q^m)_\infty^2}{(q^p;q^p)_\infty^{2}(q^{2m};q^{2m})_\infty}\pmod{p}\\
    &=\frac{(q^{2p};q^{2p})_\infty}{(q^p;q^p)_\infty^{2}}\varphi(-q^m)\\
    &=\frac{(q^{2p};q^{2p})_\infty}{(q^p;q^p)_\infty^{2}}\left( 1+2\sum_{n=1}^\infty (-1)^n q^{mn^2} \right)\pmod{p}.
\end{align*}
Since $pn+r\equiv mk^2\pmod{p}$ if and only if $r\equiv mk^2\pmod{p}$, and since we assumed that $\left( \frac{m^{-1}r}{p} \right)=-1$, we have the required congruence.
\end{proof}


\section{Proof of Theorems \ref{Theorem4}, \ref{arith.dens} and \ref{dens.op}}\label{section3}

\begin{proof}[Proof of Theorem \ref{Theorem4}]
    By setting $m=2^\alpha$, $r=1$, and $s=2$ in (\ref{gen}), we get
   \begin{align}\label{t4p1}
        \sum_{n=0}^{\infty}\overline{\mathfrak{C}}_{2^\alpha,1,2}(n)q^n=\frac{(q^2;q^2)_\infty^2(q^{2^\alpha};q^{2^\alpha})_\infty^2}{(q;q)_\infty^4(q^{2^{\alpha+1}};q^{2^{\alpha+1}})_\infty}.
\end{align} 
We use the following two dissection formula
 \begin{align}
    \label{5}
        \frac{1}{(q;q)_\infty^4}&=\frac{(q^4;q^4)_\infty^{14}}{(q^{2};q^{2})_\infty^{14}(q^8;q^8)_\infty^4}+4q\frac{(q^4;q^4)_\infty^2(q^{8};q^{8})_\infty^4}{(q^2;q^2)_\infty^{10}},
    \end{align}
which is a consequence of dissection formulas of Ramanujan \cite{ramanujan3}.
Using (\ref{5}) in (\ref{t4p1}), we get
\begin{align}\label{t4p2}
\sum_{n=0}^{\infty}\overline{\mathfrak{C}}_{2^\alpha,1,2}(n)q^n=\frac{(q^2;q^2)_\infty^2(q^{2^\alpha};q^{2^\alpha})_\infty^2}{(q^{2^{\alpha+1}};q^{2^{\alpha+1}})_\infty}\left( \frac{(q^4;q^4)_\infty^{14}}{(q^{2};q^{2})_\infty^{14}(q^8;q^8)_\infty^4}+4q\frac{(q^4;q^4)_\infty^2(q^{8};q^{8})_\infty^4}{(q^2;q^2)_\infty^{10}} \right).
\end{align}
 Now extracting the odd powers of $q$ from both sides of (\ref{t4p2}) and using binomial theorem, we get
\begin{align}\label{t4p3}
     \sum_{n=0}^{\infty}\overline{\mathfrak{C}}_{2^\alpha,1,2}(2n+1)q^n\equiv4 (q^2;q^2)_\infty^6 \pmod{8}.
\end{align}
Again, by extracting terms containing $q^{2n}$ from (\ref{t4p3}), we get

\begin{align*}
    \sum_{n=0}^{\infty}\overline{\mathfrak{C}}_{2^\alpha,1,2}(4n+1)q^{n} \equiv 4 ~(q;q)_\infty^6 \pmod{8},
\end{align*}
from which we obtain
\begin{align}\label{t4p5}
    \sum_{n=0}^{\infty}\overline{\mathfrak{C}}_{2^\alpha,1,2}(4n+1)q^{4n+1} \equiv 4 ~\eta(4z)^6 \pmod{8}.
\end{align}
Making use of Theorem \ref{RTHM1}, we can see that $\eta(4z)^6\in M_3(\Gamma_0(16),\left( \frac{-4^6}{d} \right))$. Thus $\eta(4z)^6$ has a Fourier expansion, and consider
\begin{align*}
    \eta(4z)^6= q-6q^5+9q^{9}+10q^{13}+\cdots =\sum_{n=0}^{\infty}b(n)q^n.
\end{align*}
We note that $b(n)=0$ if $n\not\equiv 1 \pmod{4}$, and for all $n\geq 0$ we have
\begin{align}\label{t4p6}
  \overline{\mathfrak{C}}_{2^\alpha,1,2}(4n+1)\equiv 4~b(4n+1) \pmod{8}.
\end{align}
Also from \cite{martin1996multiplicative}, we see that $\eta(4z)^6$ is a Hecke eigenform. Using (\ref{H1}) and (\ref{H2}), we get
\begin{align*}
    \eta(4z)^6\mid T_p&=\sum\limits_{n=0}^{\infty}\left( b(pn)+p^2 \left( \frac{-4^6}{p} \right)b\left(\frac{n}{p}\right) \right)q^n\\
    &=\lambda(p)\sum\limits_{n=0}^{\infty}b(n)q^n.
\end{align*}
Comparing the coefficients of $q^n$ on both sides of the above equation, we get
\begin{align}\label{t4p7}
   b(pn)+p^2\left( \frac{-4^6}{p} \right) b\left(\frac{n}{p}\right)=\lambda(p)b(n) .
\end{align}
Since $b(1)=1$ and $b\left(\dfrac{1}{p}\right)=0$, by putting $n=1$ in (\ref{t4p7}) we have $b(p)=\lambda(p)$. Since for $p\not \equiv 1\pmod{4}$, $b(p)=0$, we will also get $\lambda(p)=0.$ Thus for $p\not \equiv 1\pmod{4}$,
\begin{align}\label{t4p8}
   b(pn)+p^2 \left( \frac{-4^6}{p} \right) b\left(\frac{n}{p}\right)=0 .
\end{align}

Now we consider two cases; one is when $p\nmid n$ and the other is when $p\mid n$. If $p\nmid n$, substituting $n\mapsto pn+r$ in (\ref{t4p8}), where $0<r<p$, we get
\begin{align}\label{t4p9}
    b(p^2n+rp)=0.
\end{align}
If $p\mid n$, by replacing $n\mapsto pn$ in (\ref{t4p8}), we obtain
\begin{align}\label{t4p10}
    b(p^2n)=-p^2 \left( \frac{-4^6}{p} \right) b(n).
\end{align}
Now substituting $n$ by $4n-pr+1$ in (\ref{t4p9}), we get
\begin{align*}
    b(4p^2 n+p^2 +pr(1-p^2))=0,
\end{align*}
which by using (\ref{t4p6}) gives
\begin{align}\label{t4p11}
    \overline{\mathfrak{C}}_{2^\alpha,1,2}(4p^2 n+p^2 +pr(1-p^2))\equiv 0 \pmod{8}.
\end{align}
Similarly, substituting $n$ by $4n+1$ in (\ref{t4p10}) and by utilizing (\ref{t4p6}), we have
\begin{align}\label{t4p12}
   \overline{\mathfrak{C}}_{2^\alpha,1,2}(4p^2 n+p^2)\equiv -p^2\left( \frac{-4^6}{p} \right) \overline{\mathfrak{C}}_{2^\alpha,1,2}(4n+1) \pmod{8}.
\end{align}
Since $\gcd\left(\dfrac{1-p^2}{4},p\right)=1,$ when $r$ runs over a residue system excluding the multiples of $p$, so does $\dfrac{(1-p^2)r}{4}$. Thus for $p\nmid j$, we can rewrite (\ref{t4p11}) as
\begin{align}\label{t4p13}
     \overline{\mathfrak{C}}_{2^\alpha,1,2}(4p^2 n+p^2 +4pj)\equiv 0 \pmod{8}.
\end{align}
Consider the primes $p_i$ such that $p_i\not \equiv 1 \pmod{4}$, for $1\leq i \leq k+1$. Since 
\begin{align*}
    4p_1^2\cdots p_k^2n+p_1^2\cdots p_k^2 = 4p_1^2\left( p_2^2\cdots p_k^2n+\frac{p_2^2\cdots p_k^2-1}{4} \right)+p_1^2 ,
\end{align*}
by repeatedly using (\ref{t4p12}), we obtain that
\begin{align}\label{t4p14}
    \overline{\mathfrak{C}}_{2^\alpha,1,2}( 4p_1^2\cdots p_k^2n+p_1^2\cdots p_k^2)\equiv (-1)^k \left( \frac{-4^6}{p} \right)^k p_1^2\cdots p_k^2~\overline{\mathfrak{C}}_{2^\alpha,1,2}(4n+1) \pmod{8}.
\end{align}
Let $j\not \equiv 0 \pmod{p_{k+1}}$. Then by taking $n$ as $p_{k+1}^2 n+\dfrac{p_{k+1}^2-1}{4}+p_{k+1}j$ in (\ref{t4p14}), and using (\ref{t4p13}),  yields (\ref{EQ1}), which completes the proof.
\end{proof}

\begin{proof}[Proof of Theorem \ref{arith.dens}]
First we define 
    \begin{align*}
        A(z):=\frac{\eta(z)^p}{\eta(pz)}.
    \end{align*}
Using binomial theorem, we have 
\begin{align*}
    A(z)\equiv 1\pmod{p}.
\end{align*}
    Now let us define 
    \begin{align*}
    F(z)&:=\frac{\eta(2pz)^{r-p+1}}{\eta(pz)^{2(r-p+1)}}A(z)\\
  &=\frac{\eta(2pz)^{r-p+1}\eta(z)^{p}}{\eta(pz)^{2(r-p+1)+1}},
\end{align*}
and combining with the above congruence modulo $p$, we get
\begin{align}\label{N1}
F(z)   \equiv  \sum_{n=0}^{\infty}\overline{\mathfrak{C}}_{p,r,p}(n)q^{n} \pmod{p}.
\end{align}
Under the stated assumption, we can see that $F(z)$ is an eta-quotient with $N=24p$, having the weight $w=\dfrac{2p-(r+2)}{2}$ which is a positive integer. Moreover, $F(z)$ satisfies all the conditions of Theorem \ref{RTHM1} with the corresponding character given by
\begin{align}\label{charp}
    \chi(\bullet)=\left( \frac{(-1)^{w}(2p)^{r-p+1}(p)^{-\left(2(r-p+1)+1\right)}}{\bullet} \right).
\end{align}
By Theorem \ref{RTHM2}, $F(z)$ is holomorphic at a cusp $\dfrac{c}{d}$ of $\Gamma_0(24p)$, with $d\mid 24p$, if and only if
\begin{align*}
   S= (r-p+1)\frac{\gcd(d,2p)^2}{2p}+ p-(2(r-p+1)+1)\frac{\gcd(d,p)^2}{p}\geq 0.
\end{align*}
Now we consider all the possible cases arising from the given conditions and systematically evaluate the corresponding values of $S$ for each case.

If $d=1$ or $3$, we get $S=\dfrac{-3(r-p+1)}{2p}+p-\dfrac{1}{p}.$ When $d=2^k$ or $3\cdot2^k$, for $k=1,2$ and $3$, we have $S=p-\dfrac{1}{p}$. For $d=p$ or $3p$, we calculate that $S=\dfrac{-3p(r-p+1)}{2}$, and for $d=2^k\cdot p$ or $3\cdot2^k\cdot p$, where $k=1,2$ and $3$, we get $S=0$.

In all the possible cases considered above, we obtain $S\geq 0$. Therefore, by the criterion for eta-quotients, $F(z)$ is a modular form on $\Gamma_0(24p)$ with Nebentypus character given by (\ref{charp}). We now apply Theorem \ref{RTHM3}, i.e.,
if $f(z)\in M_w(\Gamma_0(N),\chi)$ has a Fourier expansion
\begin{align*}
    f(z)=\sum_{n=0}^{\infty}c(n)q^n \in \mathbb{Z}[[q]], 
\end{align*}
then we can find a constant $\alpha>0$ such that 
\begin{align*}
    \#\{n\leq X:c(n)\not\equiv 0\pmod{v}\}=O\left( \frac{X}{(\log{X})^\alpha} \right).
\end{align*}
Consequently,
\[\lim_{X\to \infty}\frac{\#\{n\leq X: c(n)\not\equiv 0\pmod{v}\}}{X}=0.\]
Hence, the complementary set $\{n\leq X: c(n)\equiv 0\pmod{v}\}$ has arithmetic density $1$. Since $F(z)\in M_{w}(\Gamma_0(24p),\chi)$, the Fourier coefficients of $F(z)$ are almost always divisible by $v=p.$ Hence by using (\ref{N1}), we can see that Theorem \ref{arith.dens} holds.
\end{proof}

The $U_d$ operator introduced by Atkin plays a fundamental role in the theory of modular functions and is defined through its action on their Fourier expansions. 
\begin{definition}
    For a positive integer $d$, the $U_d$ operator is defined by
    \begin{align*}
        U_d\left( \sum_{n\geq n_0}a(n)q^n \right):= \sum_{dn\geq n_0}a(dn)q^n.
    \end{align*}
\end{definition}

As a direct consequence of the definition of the $U_d$ operator, for any $q$-series $H(q)$, we have the identity
\begin{align*}
     U_d\left( H(q^d) \right)=H(q).
\end{align*}
Keeping these facts in mind, we prove Theorem \ref{dens.op}.
\begin{proof}[Proof of Theorem \ref{dens.op}]

Let \begin{align*}
    F_{m,r,s}(q):=\sum_{n=0}^{\infty}\overline{\mathfrak{C}}_{m,r,s}(n)q^n=\frac{(q^2;q^2)_\infty^{s}(q^{2m};q^{2m})_{\infty}^{r-s}}{(q^m;q^m)_\infty^{2(r-s)}(q;q)_\infty^{2s}}.
\end{align*}
Then, 
\begin{align}
  \notag F_{m,r,s}(q)&\equiv \frac{(q^{2p};q^{2p})_\infty^{\frac{s}{p}}(q^{2pm};q^{2pm})_{\infty}^{\frac{r-s}{p}}}{(q^{pm};q^{pm})_\infty^{\frac{2(r-s)}{p}}(q^p;q^p)_\infty^{\frac{2s}{p}}}\pmod{p}\\
 \label{deq1}   &=F_{m,\frac{r}{p},\frac{s}{p}}(q^p)\pmod{p}.
\end{align}
Now we apply the $U_p$ operator on both sides of (\ref{deq1}), and we obtain
\begin{align*}
    U_p\left( F_{m,r,s}(q) \right) \equiv U_p\left( F_{m,\frac{r}{p},\frac{s}{p}}(q^p) \right)\pmod{p}\\
    = F_{m,\frac{r}{p},\frac{s}{p}}(q)\pmod{p}.
\end{align*}
This implies that
\begin{align}\label{deq2}
    \overline{\mathfrak{C}}_{m,r,s}(pn)\equiv \overline{\mathfrak{C}}_{m,\frac{r}{p},\frac{s}{p}}(n) \pmod{p}.
\end{align}
   Now substitute $n\mapsto pn$ in (\ref{deq2}) to obtain 
   \begin{align*}
         \overline{\mathfrak{C}}_{m,r,s}(p^2n)\equiv \overline{\mathfrak{C}}_{m,\frac{r}{p^2},\frac{s}{p^2}}(n) \pmod{p}.
   \end{align*}
Repeating this process $\alpha-1 $ times will give 
\begin{align*}
      \overline{\mathfrak{C}}_{m,r,s}(p^\alpha n)\equiv \overline{\mathfrak{C}}_{m,\frac{r}{p^\alpha},\frac{s}{p^\alpha}}(n)\pmod{p}.
\end{align*}
Now define 
\begin{align*}
    A(X):=\# \{0<n\leq X:\overline{\mathfrak{C}}_{m,r,s}(n)\equiv 0 \pmod{p} \},
\end{align*}
and 
\begin{align*}
     B(X):=\# \{0<n\leq X:n\not\equiv 0 \pmod{p} \}.
\end{align*}
Since we have $ \overline{\mathfrak{C}}_{m,r,s}(pn)\equiv \overline{\mathfrak{C}}_{m,\frac{r}{p},\frac{s}{p}}(n) \pmod{p}$ from (\ref{deq2}),  we can write if $p\nmid n$
\begin{align*}
    \overline{\mathfrak{C}}_{m,r,s}(n)\equiv 0\pmod{p}.
\end{align*}
Consequently, we get $A(X)\geq B(X)$.
Also, we have 
\begin{align*}
    \lim_{X\to \infty}\frac{B(X)}{X}=\lim_{X\to \infty}\frac{X-\lfloor\frac{X}{p}\rfloor}{X}=1-\frac{1}{p}.
\end{align*}
Then it follows that 
\begin{align*}
   \lim_{X\to \infty}\frac{A(X)}{X} \geq 1-\frac{1}{p}.
\end{align*}
This completes the proof.
\end{proof}


\section{Proof of Theorem \ref{thm2}}\label{section4}

Theorem \ref{thm2} is established using Tauberian methods by Ingham~\cite{ingham1941tauberian}.

\noindent Consider the infinite product 
\[f(\tau)=\prod_{n=1}^{\infty}(1-q^n)^{-a_n},\]
where $q=e^{-\tau}$ and Re $\tau>0$ (or equivalently $|q|<1$). We assume that $a_n$ are nonnegative real numbers. We also consider the associated Dirichlet series
\[D(s)=\sum_{n=1}^{\infty}\frac{a_n}{n^s}\text{ \hspace{1cm}}(s=\sigma+it),\]
which is assumed to converge for $\sigma >\alpha$, where $\alpha$ is a positive real number. We further assume that $D(s)$ possesses an analytic continuation in the region $\sigma \geq -C_0~(0<C_0<1)$ and that in this region $D(s)$ is analytic except for a pole of order $1$ at $s=\alpha$ with residue $A$. Finally, suppose that 
\[D(s)=O(|t|^{C_1})\]
uniformly in $\sigma\geq -~C_0$ as $|t|\to \infty,$ where $C_1$ is a fixed positive real numbers.

\begin{lemma}[Lemma $6.1$, \cite{andrews1998theory}]\label{lem*}
    Under these assumptions (with $\tau=y+2\pi ix$),
    \[f(\tau)=\exp{[A\Gamma(\alpha)\zeta(\alpha+1)\tau^{-\alpha}-D(0)\log{\tau}+D'(0)+O(y^{C_0})]}\]
    uniformly in $x$ as $y\to 0,$ provided $|\arg{\tau}|\leq \pi/4,~|x|\leq \dfrac{1}{2}.$
\end{lemma}

\begin{theorem}\label{ingham}(Ingham).
    Let $f(q)=\sum\limits_{n=0}^{\infty}a(n)q^n$ be a power series with weakly increasing nonnegative coefficients and radius of convergence equal to $1$. If there exists $\alpha>0,~\beta,~\gamma\in \mathbb{R}$ such that
    \begin{align*}
        f(e^{-t})\sim \beta t^{\gamma} \exp{\left(\frac{\alpha}{t}\right)}, ~~~~t\to 0^{+},
    \end{align*}
    then 
    \begin{align*}
        a(n)\sim \frac{\beta}{2\sqrt{\pi}}\frac{\alpha^{\frac{\gamma}{2}+\frac{1}{4}}}{n^{\frac{\gamma}{2}+\frac{3}{4}}}\exp{(2\sqrt{\alpha n})}, ~~~~ n\to \infty.
    \end{align*}
\end{theorem}

\begin{proof}[Proof of Theorem \ref{thm2}]

From (\ref{gen}), we have
\begin{align*}
       F(q)=  \sum_{n=0}^{\infty}\overline{\mathfrak{C}}_{m,r,s}(n)q^n=\frac{(q^2;q^2)_\infty^{s}(q^{2m};q^{2m})_{\infty}^{r-s}}{(q^m;q^m)_\infty^{2(r-s)}(q;q)_\infty^{2s}}.
\end{align*}
Now we define 
\begin{align*}
    G_a(q):=\frac{1}{(q^a;q^a)_\infty}=\prod_{n=0}^\infty (1-q^{an+a})^{-1},
\end{align*}
and the Dirichlet series
\begin{align*}
    D_a(s):=\sum_{n=0}^\infty \frac{1}{(a+an)^s},
\end{align*}
and it is possible to write as $D_a(s)=\dfrac{1}{a^s}\zeta(s)$, where $\zeta(s)=\sum\limits_{n=1}^\infty \dfrac{1}{n^s}$ is the Riemann zeta function. By considering the Legendre series expansion of $\zeta(s)$, we find that $D_a(s)$ has a simple pole at $s=1$ with residue $\dfrac{1}{a}$. Moreover, using the properties of the Riemann zeta function, we have
\begin{align*}
    D_a(0)=\dfrac{-1}{2} \text{ and the derivative at $0$, } D_a'(0)=\log \left( \sqrt{\frac{a}{2\pi}} \right).
\end{align*}

 By using Lemma \ref{lem*}, we get as $t\to 0^+$,

\begin{align*}
    G_a(e^{-t})\sim t^{\frac{1}{2}}\sqrt{\frac{a}{2\pi}}\exp{\left( \frac{\pi^2}{6at} \right)}.
\end{align*}

 Since 
\begin{align*}
    F(q)&=\frac{G_m(q)^{2(r-s)}G_1(q)^{2s}}{G_2(q)^sG_{2m}(q)^{r-s}}\\
    &\sim \frac{\left(t^{\frac{1}{2}}\sqrt{\frac{m}{2\pi}}\exp{\left( \frac{\pi^2}{6mt} \right)}\right)^{2(r-s)}\left(t^{\frac{1}{2}}\sqrt{\frac{1}{2\pi}}\exp{\left( \frac{\pi^2}{6t} \right)} \right)^{2s}}{\left( t^{\frac{1}{2}}\sqrt{\frac{2}{2\pi}}\exp{\left( \frac{\pi^2}{12t} \right)} \right)^{s}\left( t^{\frac{1}{2}}\sqrt{\frac{2m}{2\pi}}\exp{\left( \frac{\pi^2}{12mt} \right)} \right)^{r-s}},
\end{align*}
by simplification, we will obtain

\begin{align*}
    F(e^{-t})\sim \frac{m^{\frac{r-s}{2}}}{2^r \pi^{\frac{r}{2}}} t^{\frac{r}{2}} \exp{\left( \frac{\pi^2}{4t}\left( \frac{r-s}{m}+s\right) \right)}.
\end{align*}
It is clear that $\overline{\mathfrak{C}}_{m,r,s}(n)$ is weakly increasing. Now by applying Ingham's Theorem \ref{ingham} with $\alpha=  \dfrac{\pi^2}{4}\left( \dfrac{r-s}{m}+s\right),$ $\beta=\dfrac{m^{\frac{r-s}{2}}}{2^r \pi^{\frac{r}{2}}}$, and $\gamma=\dfrac{r}{2}$, we get Theorem \ref{thm2}.    
\end{proof}


\section*{Acknowledgements}

 The first author gratefully acknowledges the Council of Scientific and Industrial Research (CSIR), Government of India, for the financial support.


\bibliographystyle{plain}
\bibliography{biregular.bib}

\end{document}